\documentclass[11pt]{amsart}

\usepackage{amsmath,amssymb,amsthm,amscd}
\usepackage[shortlabels]{enumitem}
\usepackage{mathtools}

\usepackage{tikz-cd}

\usepackage[alphabetic,initials]{amsrefs}

\usepackage{hyperref}

\hypersetup{
colorlinks=true,
linkcolor={blue},
citecolor={red!50!black}
}

\usepackage{comment}

\newtheorem{theorem}{Theorem}[section]
\newtheorem*{theorem*}{Theorem}
\newtheorem{lemma}[theorem]{Lemma}
\newtheorem*{lemma*}{Lemma}
\newtheorem{corollary}[theorem]{Corollary}

\newtheorem{proposition}[theorem]{Proposition}
\newtheorem{definition}[theorem]{Definition}
\newtheorem{remark}[theorem]{Remark}
\newtheorem{example}[theorem]{Example}

\theoremstyle{definition}

\newtheorem{thm}{Theorem}[section]

\newtheorem*{cor}{Corollary D}

\numberwithin{equation}{section}

\newcommand{\cG}{\mathcal G}
\newcommand{\cH}{\mathcal H}

\def\Bz{\mathbb{B}}

\def\Iz{\mathbb{I}}

\def\Nz{\mathbb{N}}

\def\Rz{\mathbb{R}}

\def\Zz{\mathbb{Z}}

\def\1z{\mathbb{1}}
\newcommand{\acts}{\curvearrowright}

\DeclareMathOperator{\supp}{supp}

\newcommand{\Ind}{\textup{Ind}}

\newcommand{\spn}{\textup{span}}

\renewcommand{\max}{\textup{max}}

\newcommand{\Max}{\textup{Max}}

\newcommand{\Coind}{\textup{Coind}}

\newcommand{\inn}[2]{\langle{#1},{#2}\rangle}

\begin{document}

\title[Maximal ideals of reduced group C*-algebras]{Maximal ideals of reduced group C*-algebras and Thompson's groups}

\author[K. A. Brix]{Kevin Aguyar Brix}
\address[K. A. Brix]{Centre for Mathematical Sciences, Lund University, Box 118, 221 00 Lund, Sweden}
\curraddr{Department of Mathematics and Computer Science, University of Southern Denmark, 5230 Odense, Denmark}
\email{kabrix.math@fastmail.com}

\author[C. Bruce]{Chris Bruce}
\address[C. Bruce]{School of Mathematics, Statistics and Physics, Herschel Building, Newcastle University, Newcastle upon Tyne, NE1 7RU, UK}
\email[Bruce]{chris.bruce@newcastle.ac.uk}

\author[K. Li]{Kang Li}
\address[K. Li]{Department of Mathematics, University of Erlangen–Nuremberg, Cauerstraße 11, 91058 Erlangen, Germany}
\email{kang.li@fau.de}

\author[E. Scarparo]{Eduardo Scarparo}
\address[E. Scarparo]{Center for Engineering, Federal University of Pelotas, R. Benjamin Constant 989,
96010-020, Pelotas/RS, Brazil}
\email{eduardo.scarparo@ufpel.edu.br}

\thanks{K. A. Brix was supported by a DFF-international postdoc (case number 1025-00004B) and a Starting grant from the Swedish Research Council (2023-03315).
C. Bruce has received funding from the European Union’s Horizon 2020 research and innovation programme under the Marie Sklodowska-Curie grant agreement No 101022531 and the European Research Council (ERC) under the European Union’s Horizon 2020 research and innovation programme (grant agreement No. 817597). 
This research was partially supported by the London Mathematical Society through a Collaborations with Developing Countries Grant (Ref No: 52218).}

\maketitle
\begin{abstract}
Given a conditional expectation $P$ from a C*-algebra $B$ onto a C*-subalgebra $A$, we observe that induction of ideals via $P$, together with a map which we call co-induction, forms a Galois connection between the lattices of ideals of $A$ and $B$. Using properties of this Galois connection, we show that, given a discrete group $G$ and a stabilizer subgroup $G_x$ for the action of $G$ on its Furstenberg boundary, induction gives a bijection between the set of maximal co-induced ideals of $C^*(G_x)$ and the set of maximal ideals of $C^*_r(G)$. 

As an application, we prove that the reduced C*-algebra of Thompson's group $T$ has a unique maximal ideal. Furthermore, we show that, if Thompson's group $F$ is amenable, then $C^*_r(T)$ has infinitely many ideals. 
\end{abstract}

\section{Introduction}

A discrete group $G$ is said to be \emph{C*-simple} if its reduced C*-algebra $C^*_r(G)$ is simple. In \cite{KK17}, Kalantar and Kennedy discovered a surprising connection between C*-simplicity and topological dynamics: they showed that $G$ is C*-simple if and only if the action of $G$ on its Furstenberg boundary $\partial_F G$ is free (that is, all stabilizer subgroups $G_x$ are trivial). 

Recently, Christensen and Neshveyev proved in \cite[Theorem A]{CN24} that, given an \'{e}tale groupoid $\cG$, any proper ideal of $C^*_r(\cG)$ is contained in a proper ideal induced from an isotropy group of $\cG$. 

Inspired by these results, given a group $G$, we study the relationship between ideals in $C^*_r(G)$ and ideals in $C^*(G_x)$, for $x\in\partial_F G$ (recall that $G_x$ is amenable for all $x\in\partial_FG$ by \cite[Proposition~2.7]{BKKO17}).

Let $P$ be a conditional expectation from a C*-algebra $B$ onto a C*-subalgebra $A$. We call the map that takes an ideal $I$ of $B$ into the ideal $\overline{P(I)}$ of $A$ \emph{co-induction}. Motivated by an observation of Kwa\'{s}niewski and Meyer in a related setting \cite[Section 2]{KM20}, we show that induction of ideals via $P$ in the sense of Rieffel (Definition~\ref{def:induced}), together with co-induction, forms a Galois connection between the lattices of ideals of $A$ and $B$. Using properties of this Galois connection, we prove the following:

\begin{thm}[Theorem \ref{thm:maxideals}]
Let $P$ be a conditional expectation from a unital C*-algebra $B$ onto a unital C*-algebra $A$. The following conditions are equivalent:
\begin{enumerate}[\upshape(1)]
    \item for every proper ideal $J$ of $B$, the ideal $\overline{P(J)}$ of $A$ is proper;
    \item every maximal ideal of $B$ is induced from a maximal ideal of $A$;
    \item induction defines a bijection between the set of maximal co-induced ideals of $A$ and the set of maximal ideals of $B$. 
\end{enumerate}
\end{thm}

Combining this with \cite[Lemma 7.9]{BKKO17}, we obtain the following:

\begin{thm}[Theorem \ref{thm:mi}]\label{thm:thm}
    Let $G$ be a group and $x\in\partial_F G$. Then,
\begin{enumerate}[\upshape(1)]
    \item every maximal ideal of $C^*_r(G)$ is induced from a maximal ideal of $ C^*(G_x)$;
    \item induction gives a bijection between the set of maximal co-induced ideals of $C^*(G_x)$ and the set of maximal ideals of $C^*_r(G)$. 
\end{enumerate}
\end{thm}
Since the isotropy groups $G_x$ are amenable, Theorem~\ref{thm:thm} has the following surprising consequence: it reduces the problem of understanding maximal ideals in reduced group C*-algebras of arbitrary groups to understanding ideals in group C*-algebras of \emph{amenable} groups.

Given a countable group $G$, Le Boudec and Matte Bon proved in \cite[Section 2.4]{LBMB18} that $\mathcal{A}_G:=\{G_x:x\in \partial_F G\}$ is the largest amenable uniformly recurrent subgroup of $G$, and used this fact for computing $\mathcal{A}_G$ for several examples of groups. Thus, Theorem \ref{thm:thm} can also be seen as the version for ideals of the fact shown by Breuillard, Kalantar, Kennedy and Ozawa in \cite[Theorem 4.1]{BKKO17} that any tracial state on $C^*_r(G)$ is supported on the largest amenable normal subgroup of $G$ . 

The Effros--Hahn conjecture states that every primitive ideal in the crossed product of a commutative C*-algebra by a locally compact group should be induced from a primitive ideal in the C*-algebra of some stabilizer subgroup. For the case of discrete amenable groups, the conjecture was confirmed by Sauvageot in \cite[Théorème 6.6]{S77}, and since then it has been extended to various other contexts (see \cite[Section 1]{DE17} for further references).

In Example \ref{ex:cs}, we show that a version of the Effros--Hahn conjecture in the setting of Theorem \ref{thm:thm} does not hold, by presenting an example of a group $G$ such that $C^*_r(G)$ has a primitive ideal which is not induced from $C^*(G_x)$ for any $x\in\partial_FG$. The construction uses wreath products.

Let $F\leq T$ be Thompson's groups. It is a long-standing open problem to determine whether or not $F$ is amenable (see \cite{CFP96} for an introduction to Thompson's groups). Haagerup and Olesen showed in \cite[Theorem 5.5]{HO17} that, if $F$ is amenable, then $C^*_r(T)$ is not simple. Furthermore, in \cite[Corollary 4.2]{LBMB18}, Le Boudec and Matte Bon showed the converse. In particular, if $F$ is not amenable, then $C^*_r(T)$ has a unique maximal ideal (namely, the zero ideal). By using Theorem~\ref{thm:thm} and a description due to Dudko and Medynets (\cite{DM14}) of extreme characters on Thompson's groups, we prove the following:
\begin{thm}[Theorem \ref{thm:th}]
\label{thm:C}
 The reduced C*-algebra of Thompson's group $T$ always has a unique maximal ideal.
\end{thm}
We also strengthen the result of Haagerup and Olesen (\cite[Theorem 5.5]{HO17}) by showing that, if $F$ is amenable, then $C^*_r(T)$ has infinitely many ideals (Theorem~\ref{thm:infinite-ideals}).

These results together yield the following characterizations of amenability of $F$ in terms of the ideal structure of $C^*_r(T)$:
\begin{cor}[{Corollary \ref{cor:is}, \cite{HO17}, \cite{LBMB18}, \cite{KS22}}]
    The following statements are equivalent:
    \begin{enumerate}[\upshape(1)]
        \item Thompson's group $F$ is amenable;
        \item Thompson's group $T$ is not $C^*$-simple;
        \item the unique maximal ideal of $C^*_r(T)$ is nonzero;
        \item the unique nonzero simple quotient of $C^*_r(T)$ is traceless;
        \item $C^*_r(T)$ has infinitely many ideals.
    \end{enumerate}
\end{cor}

\subsection*{Acknowledgement}
We thank Narutaka Ozawa for allowing us to include Example \ref{ex:oz}.

\section{Induction of ideals via conditional expectations}
Let $A$ be a C*-subalgebra of a C*-algebra $B$, and suppose we have a conditional expectation $P\colon B\to A$, which will be fixed throughout this section.
We let $\Iz(A)$ and $\Iz(B)$ denote the complete lattices of (closed and two-sided) ideals of $A$ and $B$, respectively. 
The meet and join of a family of ideals $(I_s)_{s\in S}$ are defined as 
\[
\bigwedge_{s\in S} I_s\coloneqq \bigcap_{s\in S}I_s,
\]
and 
\[
\bigvee_{s\in S}I_s\coloneqq \overline{\spn}(\{a : a\in I_s, s\in S\}),
\]
respectively.

The following notion of induction is obtained by specialization from the usual notion of induction of an ideal by an imprimitivity bimodule (see, for instance, \cite[Equation (2.5.3)]{CELY17}).

\begin{definition}
\label{def:induced}
 Let $\Ind_P\colon \Iz(A)\to\Iz(B)$ be given by
\[
\Ind_P(I)\coloneqq \{b\in B :P(xby)\in I \text{ for all } x,y\in B\},
\]
for all $I\in\Iz(A)$. We call $\Ind_P(I)$ the ideal of $B$ \emph{induced} from $I$ via $P$.    
\end{definition}

We will now explain the term ``induced ideal''. Given a nondegenerate representation $\pi\colon A\to \Bz(\cH)$, there is a Hilbert space $B\otimes_A\cH$ endowed with a linear map with dense image from the algebraic tensor product $B\odot_A\cH$ into $B\otimes_A\cH$, and such that the inner product of $B\otimes_A\cH$ satisfies
\[
\inn{x\otimes \xi}{y\otimes \eta}=\inn{\pi(P(y^*x))\xi}{\eta}_\cH,
\]
for all $x,y\in B$ and $\xi,\eta\in\cH$. 
Furthermore, there is a representation $\Ind_P(\pi)\colon B\to\Bz(B\otimes_A\cH)$ given by $\Ind_P(\pi)(b)(x\otimes \xi)=(b x)\otimes \xi$, for all $b,x\in B$ and $\xi\in \cH$, see \cite[Theorem 1.8]{R74}. The representation $\Ind_P(\pi)$ is called the \emph{induced representation} of $\pi$ via $P$. 

The next observation justifies the terminology in Definition~\ref{def:induced}.
\begin{proposition}
Suppose $I$ is an ideal of $A$, and $\pi\colon A\to \Bz(\cH)$ is a nondegenerate representation such that $I=\ker(\pi)$. Then $\Ind_P(I)=\ker(\Ind_P(\pi))$. 
\end{proposition}
\begin{proof}
 Given $b\in B$, we have the following:
 \begin{align*}
b\in \ker(\Ind_P(\pi))&\iff (b y)\otimes\xi =0 \text{ for all }y\in B,\xi\in \cH\\
&\iff \langle (b y)\otimes\xi, x\otimes\eta\rangle =0 \text{ for all }x,y\in B,\xi,\eta\in \cH\\
&\iff \langle \pi(P(xby))\xi, \eta\rangle_\cH =0 \text{ for all }x,y\in B,\xi,\eta\in \cH\\
&\iff \pi(P(xby))=0\text{ for all }x,y\in B\\
&\iff P(xby)\in I \text{ for all }x,y\in B,
 \end{align*}
 as wanted.
\end{proof}

\begin{remark}
It follows from the polarization identity
\[
P(x^*by)=\frac{1}{4}\sum_{j=0}^3i^jP((x+i^jy)^*b(x+i^jy)),
\]
that, if $I$ is an ideal in $A$, then
\[
\Ind_P(I)= \{b\in B : P(x^*bx)\in I \text{ for all } x\in B\}.
\]
\end{remark}   

Our next aim is to show that $\Ind_P$ is part of a Galois connection, in the sense of \cite[Definition~7.23]{DP02}.

\begin{definition}
    Let $\Coind_P\colon \Iz(B)\to\Iz(A)$ be given by 
    \[
    \Coind_P(J)\coloneqq \overline{P(J)},
    \]
    for all $J\in\Iz(B)$. We call $\Coind_P(J)$ the ideal of $A$ \emph{co-induced} from $J$ via $P$.  
\end{definition}

Let $\Iz_P(A)\coloneqq \Coind_P(\Iz(B))$ and $\Iz_P(B)\coloneqq \Ind_P(\Iz(A))$. We shall call the ideals in $\Iz_P(A)$ \emph{co-induced} and the ideals in $\Iz_P(B)$ \emph{induced}.

In some cases, $P(J)$ is known to be automatically closed (see \cite[Proposition 3.22]{BFPR} and \cite[Lemma 2.1]{CN24}). In general, this is false, as the following example, due to N. Ozawa (personal communication), shows.

\begin{example}[Ozawa]\label{ex:oz}
Take a sequence $(A_i)_{i\in \Nz}$ of disjoint finite nonempty subsets of $\Nz$ such that $|A_i|\to+\infty$ and $\Nz=\bigsqcup_{i\in \Nz}A_i$.

Let $P\colon \ell^\infty(\Nz)\to\ell^\infty(\Nz)$ be given by 
\[P(f)(x):=\sum_{y\in A_i}\frac{f(y)}{|A_i|},\]
for $f\in \ell^\infty(\Nz)$, $i\in \Nz$, and $x\in A_i$. Then $P$ is a conditional expectation onto $D:=\{f\in \ell^\infty(\Nz):\text{$\forall i$, $f|_{A_i}$ is constant}\}$.

Fix a sequence $(a_i)_{i\in \Nz}\subset \Nz$ such that, for each $i\in\Nz$, $a_i\in A_i$. Let $J:=\ell^\infty(\bigsqcup_{i\in\Nz}\{a_i\})\unlhd\ell^\infty(\Nz)$. 

Given $f\in J$ and $i\in\Nz$, we have that 
\[P(f)(a_i)=\frac{f(a_i)}{|A_i|}.\]

Therefore, $\overline{P(J)}=\{g\in D:g(a_i)\to 0\}$. Let $g\in D$ be such that $g(a_i)=|A_i|^{-1/2}$, for all $i\in\Nz$. Then $g\in\overline{P(J)}\setminus P(J)$. 
\end{example}

\begin{remark} In an algebraic setting similar to ours, the ``admissible'' ideals of \cite[Definition 4.6]{DE17} correspond to what we call ``co-induced'' ideals (see \cite[Proposition 4.12]{DE17}).
\end{remark}

Let us record the following observations, which are easy consequences of the definitions.
\begin{lemma}
\label{lem:closures}
We have that
\begin{enumerate}[\upshape(1)]
\item  $\Ind_P$ and $\Coind_P$ are order-preserving;
    \item $\Coind_P(\Ind_P(I))\subseteq I$ for all $I\in\Iz(A)$;
    \item $J\subseteq \Ind_P(\Coind_P(J))$ for all $J\in\Iz(B)$.
\end{enumerate}
\end{lemma}

The next result says that the pair $(\Coind_P,\Ind_P)$ forms a Galois connection between $\Iz(B)$ and $\Iz(A)$. 
The result is an immediate consequence of Lemma~\ref{lem:closures} (see also \cite[Lemma 7.26]{DP02}). 
\begin{proposition}\label{prop:gc}
For all $I\in\Iz(A)$ and $J\in\Iz(B)$, we have
\begin{equation*}
    J\subseteq \Ind_P(I) \quad \text{ if and only if } \quad \Coind_P(J)\subseteq I.
\end{equation*}
\end{proposition}

That $(\Coind_P,\Ind_P)$ is a Galois connection between $\Iz(B)$ and $\Iz(A)$ has the following consequences (see also \cite[Proposition~2.9]{KM20}, where similar observations have been made in a related setting).

\begin{proposition}
\label{prop:Galois}
The maps $\Coind_P\colon\Iz(B)\to\Iz(A)$ and $\Ind_P\colon\Iz(A)\to\Iz(B)$ satisfy the following:
    \begin{enumerate}[\upshape(1)]
        \item $\Ind_P(I)=\Ind_P(\Coind_P(\Ind_P(I)))$ for all $I\in\Iz(A)$; \label{i:1}
        \item $\Coind_P(J)=\Coind_P(\Ind_P(\Coind_P(J)))$ for all $J\in\Iz(B)$; \label{i:2}
        \item $\Ind_P$ preserves meets and $\Coind_P$ preserves joins; \label{i:3}
        \item $\Coind_P$ and $\Ind_P$ restrict to mutually inverse isomorphisms between the complete lattices $\Iz_P(A)$ and $\Iz_P(B)$; \label{i:4}
        \item for all $I\in\Iz(A)$, $\Coind_P(\Ind_P(I))$ is the largest co-induced ideal that is contained in $I$; \label{i:5}
        \item for all $J\in\Iz(B)$, $\Ind_P(\Coind_P(J))$ is the smallest induced ideal that contains $J$; \label{i:6}
	\item given a proper ideal $I\in\Iz(A)$, we have that $\Ind_P(I)\in\Iz(B)$ is proper. \label{i:7}
    \end{enumerate}
\end{proposition}
\begin{proof}
\ref{i:1} and \ref{i:2} follow from \cite[Lemma~7.26]{DP02}.

\ref{i:3} is the content of \cite[Proposition~7.31]{DP02}.

\ref{i:4} is an immediate consequence of \ref{i:1}-\ref{i:3}.

\ref{i:5} and \ref{i:6} are straightforward consequences of Lemma \ref{lem:closures} and Proposition \ref{prop:gc}.

\ref{i:7} follows from Proposition \ref{prop:gc}.
\end{proof}

\subsection*{Maximal ideals}
Let $\Max(A)$ and $\Max(B)$ denote the maximal (proper) ideals of $A$ and $B$, respectively. 
Let $\Max_P(A)$ denote the collection of maximal co-induced (proper) ideals of $A$ via $P\colon B \rightarrow A$.

\begin{lemma}    
\label{lem:IinM}
If $A$ is unital and $I$ is a proper co-induced ideal of $A$, then there exists $M\in\Max_P(A)$ such that $I\subseteq M$.
\end{lemma}

\begin{proof}
Since $A$ is unital, the join of a chain of proper ideals of $A$ is a proper ideal. The result then follows from Zorn's lemma and the fact that, by part \ref{i:3} of Proposition \ref{prop:Galois}, $\Coind_P$ preserves joins.
\end{proof}

\begin{theorem}
\label{thm:maxideals}
Let $P\colon B\to A$ be a conditional expectation and consider the following conditions:
\begin{enumerate}[\upshape(1)]
    \item for every proper ideal $J\in\Iz(B)$, the ideal $\Coind_P(J)\in\Iz_P(A)$ is proper; \label{i:thm_1}
    \item for every proper ideal $J\in\Iz(B)$, there exists a proper ideal $I\in\Iz(A)$ such that $J\subseteq\Ind_P(I)$; \label{i:thm_2}
    \item for every $J\in \Max(B)$, there exists $I\in\Max(A)$ such that $J=\Ind_P(I)$; \label{i:thm_3}
    \item the map $\Ind_P$ defines a bijection from $\Max_P(A)$ to $\Max(B)$. \label{i:thm_4}
\end{enumerate}
Then \ref{i:thm_1}$\iff$\ref{i:thm_2}. If $A$ is unital, then \ref{i:thm_2}$\implies$\ref{i:thm_3}$\implies$\ref{i:thm_4}. If $B$ is unital, then \ref{i:thm_4}$\implies$\ref{i:thm_1}.
\end{theorem}
\begin{proof}
\ref{i:thm_1}$\implies$\ref{i:thm_2}. Let $I\coloneqq \Coind_P(J)$. Then $J\subseteq\Ind_P(I)$ by Lemma \ref{lem:closures}.

\ref{i:thm_2}$\implies$\ref{i:thm_1} follows from the fact that, by Lemma \ref{lem:closures}, 
\[
\Coind_P(J)\subset\Coind_P(\Ind_P(I))\subset I.
\]

\ref{i:thm_2}$\implies$\ref{i:thm_3}. Let $J\in\Max(B)$. Then, \ref{i:thm_2} gives us a proper ideal $I\in\Iz(A)$ such that $J\subseteq\Ind_P(I)$. Since $A$ is unital, we can assume $I$ is maximal. From part \ref{i:7} of Proposition~\ref{prop:Galois}, $\Ind_P(I)$ is a proper ideal, so that $J=\Ind_P(I)$ by maximality of $J$.

\ref{i:thm_3}$\implies$\ref{i:thm_4}. By part \ref{i:4} of Proposition~\ref{prop:Galois}, we have that $\Ind_P$ is injective on $\Max_P(A)$. 
Given $J\in\Max(B)$, we need to find $M\in\Max_P(A)$ such that $\Ind_P(M)=J$. By assumption, there exists a proper ideal $I\in\Iz(A)$ such that $J=\Ind_P(I)$. By part \ref{i:1} of Proposition~\ref{prop:Galois}, we can replace $I$ with $I' \coloneqq \Coind_P(\Ind_P(I))$, which is co-induced and satisfies $J = \Ind_P(I')$. By Lemma~\ref{lem:IinM}, there exists $M\in\Max_P(A)$ such that $I'\subseteq M$. Now we have $J=\Ind_P(I')\subseteq\Ind_P(M)$. Hence $J=\Ind_P(M)$ by maximality of $J$.

\ref{i:thm_4}$\implies$\ref{i:thm_1}. Take $M\in\Max(B)$ such that $J\subseteq M$, and let $I\in\Max_P(A)$ be such that $M=\Ind_P(I)$. 
Since $\Coind_P(J)\subseteq\Coind_P(\Ind_P(I))\subseteq I$, we conclude that $\Coind_P(J)$ is proper.
\end{proof}

\section{Reduced group C*-algebras}\label{sec:rgc}

Throughout this section, we assume that groups are discrete.
Let $G$ be a group. 
Given an action of $G$ on a compact Hausdorff space $X$ by homeomorphisms and $x\in X$, let $G_x \coloneqq \{g\in G:gx=x\}$ be the stabilizer of $x$.

The \emph{Furstenberg boundary} of $G$, denoted by $\partial_FG$, is the universal minimal strongly proximal compact $G$-space (\cite[Section III.1]{G76}). Recall that, for every $x\in \partial_FG$, the stabilizer $G_x$ is amenable (\cite[Proposition 2.7]{BKKO17}). 

We denote the canonical generators of the reduced group C*-algebra $C^*_r(G)$ of $G$ by $\{\lambda_g\}_{g\in G}$. Given a subgroup $H\leq G$, we let $P_H\colon C^*_r(G)\to C^*_r(H)$ be the canonical conditional expectation. Let $\Ind_H\colon\Iz(C^*_r(H))\to\Iz(C^*_r(G))$ and $\Coind_H\colon\Iz(C^*_r(G))\to\Iz(C^*_r(H))$ denote the induction and co-induction maps with respect to $P_H$.

We denote by $\Max_H(C^*_r(H))$ the set of maximal co-induced (proper) ideals of $C^*_r(H)$ via $P_H$.

  The next result is an immediate consequence of \cite[Lemma 7.9]{BKKO17} and Theorem \ref{thm:maxideals}. For completeness, we present a more direct proof. We will use operator algebraic properties of the Furstenberg boundary that can be found in, for example, \cite[Section 2.4]{BKKO17}. Given a group $G$ and unital $C^*$-algebras $A$ and $B$ endowed with $G$-actions, we call a $G$-equivariant unital completely positive linear map from $A$ to $B$ a \emph{$G$-map}.
\begin{theorem}\label{thm:mi}
    Let $G$ be a group and $x\in\partial_F G$. Then,
\begin{enumerate}[\upshape(1)]
    \item for every maximal ideal $J$ of $C^*_r(G)$, there exists a maximal ideal $I$ of $C^*(G_x)$ such that $J=\Ind_{G_x}(I)$;\label{1main}
    \item the map $\Ind_{G_x}$ defines a bijection from $\Max_{G_x}(C^*(G_x))$ to the set of maximal ideals of $C^*_r(G)$. \label{2main}
\end{enumerate}
\end{theorem}
\begin{proof}
\ref{1main}. Let $J$ be a maximal ideal of $C^*_r(G)$ and fix $x\in \partial_F G$. We claim that there exists a state $\sigma$ on $C^*_r(G)$ such that $\sigma|_J=0$ and, for all $g\in G\setminus G_x$, $\sigma(\lambda_g)=0$. 

By $G$-injectivity of $C(\partial_F G)$, there is a $G$-map $\varphi\colon \frac{C^*_r(G)}{J}\to C(\partial_F G)$. Let $\pi\colon C^*_r(G)\to \frac{C^*_r(G)}{J}$ be the quotient. By $G$-injectivity again, there is a $G$-map $\psi\colon C(\partial_F G)\rtimes_r G\to C(\partial_F G)$ such that $\psi|_{C^*_r(G)}=\varphi\circ\pi$. By $G$-rigidity, we have that $\psi|_{C(\partial_F G)}=\mathrm{Id}_{C(\partial_F G)}$. Let $\mathrm{Ev}_x\colon C(\partial_F G)\to \mathbb{C}$ be evaluation in $x$ and $\sigma:=\mathrm{Ev}_x\circ\varphi\circ\pi$. Clearly, $\sigma|_J=0$. Given $g\in G\setminus G_x$, take $f\in C(\partial_F G)$ such that $f(gx)=1$ and $f(x)=0$. Then

\begin{align*} 
\sigma(\lambda_g)=\sigma(\lambda_g)f(gx)=\mathrm{Ev_x}(\psi(\lambda_g))\mathrm{Ev_x}(g^{-1}f)&=\mathrm{Ev_x}(\psi(\lambda_g)g^{-1}f)\\
&=\mathrm{Ev}_x(\psi(f\lambda_g))\\
&=f(x)\sigma(\lambda_g)\\&=0,
\end{align*}
thus showing the claim. 

Let $K:=\{a\in C^*(G_x):\sigma(bac)=0\text{ for all $b,c\in C^*(G_x)$}\}$. Notice that $\sigma|_{C^*(G_x)}\circ P_{G_x}=\sigma$ and $C^*(G_x)$ is contained in the multiplicative domain of $P_{G_x}$. Given $a\in J$, $b,c\in C^*(G_x)$, and $d,e\in C^*_r(G)$, we have that 
\[\sigma(bP_{G_x}(dae)c)=\sigma(P_{G_x}(bdaec))=\sigma(bdaec)=0,\]
where the last equality follows from the fact that $\sigma|_J=0$. Therefore, $J\subset\Ind_{G_x}(K)$ (hence, by maximality, $J=\Ind_{G_x}(K)$). By taking a maximal ideal $I$ of $C^*(G_x)$ which contains $K$, we obtain that $J=\Ind_{G_x}(I)$.

\ref{2main} follows from \ref{1main} and Theorem \ref{thm:maxideals}.
\end{proof}
\begin{remark}
    Given $x\in \partial_F G$, consider the canonical conditional expectation $Q\colon C(\partial_F G)\rtimes_r G\to C^*(G_x)$. It follows from \cite[Theorem 3.4]{CN24} that induction gives a bijection between $\Max_Q(C^*(G_x))$ and $\Max(C(\partial_F G)\rtimes_r G)$. Furthermore, since $Q|_{C^*_r(G)}=P_{G_x}$ and $Q|_{C(\partial_F G)}$ is multiplicative, it is not difficult to check that, for any ideal $J$ of $C^*(G_x)$, it holds that $\Ind_Q(J)\cap C^*_r(G)=\Ind_{G_x}(J)$. Unfortunately, it is not clear whether this line of thinking could lead to a new proof of Theorem \ref{thm:mi}. We thank a referee for suggesting us to include this observation.
\end{remark}
\subsection*{Normal subgroups}
Given a normal subgroup $N\unlhd G$, the canonical conditional expectation $P_N\colon C^*_r(G)\to C^*_r(N)$ is $G$-equivariant with respect to the actions of $G$ on $C^*_r(G)$ and $C^*_r(N)$ given by conjugation. We say that an ideal $I\in\Iz( C_r^*(N))$ is \emph{$G$-invariant} if it is invariant with respect to this conjugation action.

\begin{lemma}
\label{lem:admiss=Ginv}
Let $N\unlhd G$. Then $I\in\Iz( C_r^*(N))$ is co-induced if and only if $I$ is $G$-invariant.
\end{lemma}
\begin{proof}
Suppose that $I$ is co-induced. 
By part \ref{i:2} of Proposition \ref{prop:Galois}, $I=\overline{P_N(\Ind_N(I))}$. For $g\in G$, we have 
\begin{align*}
\lambda_g\overline{P_N(\Ind_N(I))}\lambda_g^*=\overline{\lambda_gP_N(\Ind_N(I))\lambda_g^*}&=\overline{P_N(\lambda_g\Ind_N(I)\lambda_g^*)}\\
&=\overline{P_N(\Ind_N(I))},
\end{align*}
so $I=\overline{P_N(\Ind_N(I))}$ is a $G$-invariant ideal of $C_r^*(N)$. 

Now suppose $I$ is a $G$-invariant ideal of $C_r^*(N)$. 
By Lemma \ref{lem:closures}, in order to conclude that $I$ is co-induced, it suffices to show that $I\subseteq \overline{P_N(\Ind_N(I))}$. 
Let $a\in I$. Since $C_r^*(N)$ lies in the multiplicative domain of $P_N$, we have for all $g,h\in G$, 
\begin{align*}  
P_N(\lambda_g a\lambda_h)=P_N((\lambda_g a\lambda_g^*)\lambda_{gh})&=P_N((\lambda_g a\lambda_g^*))P_N(\lambda_{gh})\\
&=\begin{cases}
    (\lambda_g a\lambda_g^*)\lambda_{gh} & \text{ if }gh\in N,\\
    0 & \text{ if } gh\not\in N.
\end{cases}
\end{align*}

Since $\lambda_g a \lambda_g^*\in I$ by $G$-invariance and $\{\lambda_k :k\in G\}$ spans a dense *-subalgebra of $C_r^*(G)$, this shows that $a\in \Ind_N(I)$. 
Hence $I\subseteq\Ind_N(I)$ and $I\subseteq P_N(\Ind_N(I))$.
\end{proof}
\begin{remark}\label{rem:t}
Let $N$ be a normal subgroup of a group $G$. By \cite[Theorem 2.1]{B91}, there is a twisted action of $\frac{G}{N}$ on $C^*_r(N)$ such that $C^*_r(G)\cong C^*_r(N)\rtimes_r\frac{G}{N}$. Under this identification, it is not difficult to check that, given a $G$-invariant ideal $I\in \Iz(C^*_r(N))$, we have that 
\[
\Ind_N(I)=\ker\left(C^*_r(N)\rtimes_r\frac{G}{N}\to\frac{C^*_r(N)}{I}\rtimes_r\frac{G}{N}\right)
\]
(existence of the canonical map $C^*_r(N)\rtimes_r\frac{G}{N}\to\frac{C^*_r(N)}{I}\rtimes_r\frac{G}{N}$ follows, for example, from \cite[Proposition 21.3]{E17}). 
\end{remark}

 Recall that a group $G$ is said to be \emph{C*-simple} if $C^*_r(G)$ is simple. 
 By \cite[Theorem 3.1]{BKKO17}, $G$ is C*-simple if and only if the action $G\acts \partial_F G$ is free. 

The \emph{amenable radical} $R(G)$ of $G$ is the largest amenable normal subgroup of $G$, and it coincides with $\bigcap_{x\in\partial_F G}G_x$, see \cite[Proposition 2.8]{BKKO17} or \cite[Proposition 7]{F03}. 
Since $\partial_F G$ and $\partial_F\left(\frac{G}{R(G)}\right)$ are homeomorphic $G$-spaces (\cite[Theorem 11]{O14}), we have that $\frac{G}{R(G)}$ is C*-simple if and only if $G_x=R(G)$ for all $x\in \partial_F G$.

\begin{proposition}\label{prop:inv}
Let $G$ be a group such that $\frac{G}{R(G)}$ is C*-simple. 
Then, the induction map $\Ind_{R(G)}$ defines a bijection from maximal $G$-invariant ideals of $C^*(R(G))$ to maximal ideals of $C^*_r(G)$. 
\end{proposition}
\begin{proof}
  This follows from Theorem \ref{thm:mi} and Lemma \ref{lem:admiss=Ginv}. Alternatively, it follows from \cite[Theorem~4.3]{BK18} and Remark \ref{rem:t}.  
\end{proof}

The rest of the section is devoted to exhibiting an example of a group $G$ such that $\frac{G}{R(G)}$ is C*-simple, and a primitive ideal $I$ of $C^*_r(G)$ which is not induced from $C^*(R(G))$. 
This shows that, unfortunately, an analogue of Proposition \ref{prop:inv} (and of Theorem \ref{thm:mi}) does not hold for primitive ideals.

We start with a general lemma which will be used for producing pure states, and thus primitive ideals via the GNS representation. 
Given a C*-algebra $A$, we let $\mathrm{PS}(A)$ denote the set of pure states on $A$.

\begin{lemma}\label{lem:ps}
Let $G$ be a group acting on a compact space $X$. Given $x\in X$ and $\tau\in \mathrm{PS}(C^*_r(G_x))$, there is $\sigma\in \mathrm{PS}(C(X)\rtimes_r G)$ such that, for $f\in C(X)$ and $g\in G$, we have
\begin{align}\label{eq:sigma}
\sigma(f\lambda_g)=
\begin{cases}
f(x)\tau(\lambda_g) &\text{ if $g\in G_x$},\\
0 &\text{ otherwise.} 
\end{cases}
\end{align}
\end{lemma}

\begin{proof}
Let $\psi\colon C(X)\rtimes_r G\to C(X)\rtimes_r G_x$ be the conditional expectation given by
\begin{align*}
\psi(f\lambda_g)=
\begin{cases}
f\tau(\lambda_g) &\text{ if $g\in G_x$},\\
0 &\text{ otherwise,} 
\end{cases}
\end{align*}
and let $\phi\colon C(X)\rtimes_r G_x\to C^*_r(G_x)$ be the $*$-homomorphism given by $\phi(f\lambda_g)=f(x)\lambda_g$, for all $f\in C(X)$ and $g\in G$. 
Then, $\sigma\coloneqq\tau\circ\phi\circ\psi$ satisfies \eqref{eq:sigma}. 

Next, we show that $\sigma$ is a pure state. Suppose $\sigma=t\sigma_1+(1-t)\sigma_2$, for some $t\in(0,1)$ and states $\sigma_1,\sigma_2$. Since $\sigma|_{C(X)}$ is the point-evaluation $\mathrm{Ev}_x$, we have that $\sigma|_{C(X)}=\sigma_1|_{C(X)}=\sigma_2|_{C(X)}=\mathrm{Ev}_x$. Furthermore, since $\tau\in \mathrm{PS}(C^*_r(G_x))$ and $\sigma|_{C^*_r(G_x)}=\tau$, we have that $\sigma|_{C^*_r(G_x)}=\sigma_1|_{C^*_r(G_x)}=\sigma_2|_{C^*_r(G_x)}$.

Given $g\notin G_x$, we have that $gx\neq x$. Take $f\in C(X)$ such that $0\leq f\leq 1$, $f(x)=0$ and $f(gx)=1$. For $i=1,2$, we have that
\[
\sigma_i(\lambda_g)=\sigma_i(\lambda_g)f(gx)=\sigma_i(\lambda_g)\sigma_i(g^{-1}f)=\sigma_i(f)\sigma_i(\lambda_g)=f(x)\sigma_i(\lambda_g)=0.
\]
Thus $\sigma=\sigma_1=\sigma_2$.
\end{proof}

The idea for the next result is to use the fact (already employed in \cite[Example 6.5]{KS222}) that, if $H$ is an amenable nontrivial subgroup of a group $G$, then $c_0(G/H)\rtimes_r G\cong K(\ell^2(G/H))\otimes C^*(H)$ has a nonzero ideal $I$ such that $I\cap c_0(G/H)=\{0\}$. 

We let $\Zz_2\wr_{G/H}G\coloneqq \left(\bigoplus_{G/H}\Zz_2\right)\rtimes G$ denote the \emph{wreath product}
for the canonical action of $G$ on $G/H$.

\begin{proposition}\label{prop:ex}
Suppose $H$ is an amenable nontrivial infinite-index subgroup of a group $G$. Then, there exists a primitive ideal $I$ in $C^*_r(\Zz_2\wr_{G/H}G)$ that is not induced from $C^*(\bigoplus_{G/H}\Zz_2)$. 
\end{proposition}

\begin{proof}
Let $X\coloneqq\prod_{G/H}\{0,1\}$, and consider the action of $G$ on $X$ given by $(gx)(kH)\coloneqq x(g^{-1}kH)$, for all $g\in G$, $x\in X$, and $kH\in G/H$. By identifying $C^*(\bigoplus_{G/H}\Zz_2)$ and $C(X)$, we have that $C^*_r(\Zz_2\wr_{G/H}G)\cong C(X)\rtimes_r G $, and $P_{\bigoplus_{G/H}\Zz_2}$ is identified with the canonical conditional expectation $P\colon C(X)\rtimes_r G\to C(X)$. 

Notice that the stabilizer of $\delta_{H}\in X$ is $H$. 
Since $H$ is amenable, we can apply Lemma \ref{lem:ps} to the trivial character $C^*(H)\to\mathbb{C}$ to obtain $\sigma\in \mathrm{PS}(C(X)\rtimes_r G)$ such that 
\begin{align*}
\sigma(f\lambda_g)=
\begin{cases}
f(\delta_{H}) &\text{ if $g\in H$},\\
0 &\text{ otherwise,} 
\end{cases}
\end{align*}
for all $f\in C(X)$ and $g\in G$.

Let $\pi_\sigma$ be the GNS representation of the pure state $\sigma$. We claim that $\ker\pi_\sigma$ is not induced from $C(X)$ via $P$. Indeed, suppose that there is an open set $U\subseteq X$ such that 
\[
\ker\pi_\sigma=\Ind_P(C_0(U))=\{a\in C(X)\rtimes_r G:\forall g\in G, P(a\lambda_g)\in C_0(U)\}.
\]
By part \ref{i:6} of Proposition \ref{prop:Galois} and Lemma \ref{lem:admiss=Ginv}, we can assume that $U$ is $G$-invariant.

Let $Y\coloneqq \overline{\{\delta_{gH}\in X:g\in G\}}$ and note that $Y$ is homeomorphic to the one-point compactification of the discrete set $G/H$, where the point at infinity is the function that is constantly zero.

We claim that $Y^c=U$. Since $Y$ is $G$-invariant and $\sigma|_{C(X)}$ is the point-evaluation on $\delta_H\in Y$, we have that, given $f\in C_0(Y^c)$ and $r,s\in G$, it holds that $\sigma(\lambda_rf\lambda_s)=0$. Therefore, $C_0(Y^c)\subseteq \ker\pi_\sigma$, and $Y^c\subseteq U$. 

Suppose that $Y\cap U\neq\emptyset$. Since $U$ is open and $G$-invariant, we have that $\delta_H\in U$. Take $f\in C_0(U)$ such that $f(\delta_H)\neq 0$. Then, $\sigma(f)\neq 0$ and $f\notin \ker\pi_\sigma$, thus contradicting the fact that $C_0(U)\subseteq\Ind_P(C_0(U))= \ker\pi_\sigma$. Therefore, $Y^c=U$.

Finally, take $f\in C(X)$ such that $f(\delta_{H})=1$ and $f|_{Y\setminus\{\delta_H\}}=0$ (we can choose $f$ in this way because $\delta_H$ is an isolated point of $Y$). Fix $g\in H\setminus\{e\}$, and let us show that $f(1-\lambda_g)\in\ker\pi_\sigma\setminus \Ind_P(C_0(Y^c))$. 

1) $f(1-\lambda_g)\in \ker\pi_\sigma$: We will use the fact that $C(X)$ is in the multiplicative domain of $\sigma$. Given $r,s\in G$, we have that
\[
\sigma(\lambda_rf(1-\lambda_g)\lambda_s)=f(r^{-1}H)\sigma(\lambda_{rs}-\lambda_{rgs})=0.
\]

2) $f(1-\lambda_g)\notin \Ind_P(C_0(Y^c))$: This follows from the fact that $P(f(1-\lambda_g))=f\notin C_0(Y^c))$.

This concludes the proof that $\ker\pi_\sigma$ is not induced from $C(X)$ via $P$.
\end{proof}

\begin{example}\label{ex:cs}
Let $G$ be a C*-simple group and let $H\leq G$ be a nontrivial amenable subgroup (for example, $H=\langle g\rangle$ for some $g\in G\setminus\{e\}$). Then, the amenable radical of $K \coloneqq\Zz_2\wr_{G/H}G$ is $\bigoplus_{G/H}\Zz_2$, and $\frac{K}{R(K)}=G$ is C*-simple. Thus, a version of Proposition \ref{prop:inv} (and of Theorem \ref{thm:mi}) does not hold for primitive ideals.
\end{example}

\section{Application: Thompson's groups}\label{sec:tho}
Let $\Zz[1/2]$ be the set of dyadic rationals and let $\Omega \coloneqq [0,1)\cap\Zz[1/2]$. 
Thompson's group $T$ is the group of homeomorphisms on $\frac{\mathbb{R}}{\mathbb{Z}}$ that are piecewise linear, with finitely many breakpoints, all at $\Omega$, and slopes in $2^\mathbb{Z}$. Thompson's group $F$ is the subgroup of $T$ which stabilizes $0$. The commutator $[F,F]$ coincides with the set of elements of $T$ which fix pointwise a neighborhood of $0$ (\cite[Theorem 4.1]{CFP96}).

The next result is well-known (see, for example, \cite[Lemma 4.2]{CFP96}). For the convenience of the reader, we provide a proof.

\begin{lemma}\label{tecint}
Given $I,J\subseteq \Rz$ nondegenerate closed intervals with dyadic rational endpoints, there exists a piecewise linear homeomorphism $\varphi\colon I\to J$ with finitely many breakpoints in $\Zz[1/2]$, and slopes in $2^\Zz$. 
\end{lemma}
\begin{proof}
By considering a linear map, the conclusion is immediate if the lengths of $I$ and $J$ belong to $2^\Zz$. In the general case, since $I$ and $J$ can be decomposed as a union of intervals with length in $2^\Zz$, and each interval with length in $2^\Zz$ can be decomposed as a union of an arbitrary number of intervals with length in $2^\Zz$, the result follows.
\end{proof}

\begin{definition}[{\cite[Definition 2.5]{DM14}}]\label{def:comp}
Let $G$ be a group acting on a locally compact Hausdorff space $X$ by homeomorphisms. Given $g\in G$, let $\supp g\coloneqq\overline{\{x\in X:gx\neq x\}}$. The action $G\acts X$ is said to be \emph{compressible} if there exists a basis $\mathcal{U}$ for the topology of $X$ such that:
\begin{enumerate}[\upshape(1)]
    \item for all $g\in G$, there exists $U\in\mathcal{U}$ such that $\supp g\subseteq U$;\label{c:1}
    \item for all $U_1,U_2\in \mathcal{U}$, there exists $g\in G$ such that $g(U_1)\subseteq U_2$;\label{c:2}
    \item for all $U_1,U_2,U_3\in\mathcal{U}$ with $\overline{U_1}\cap\overline{U_2}=\emptyset$, there exists $g\in G$ such that $g(U_1)\cap U_3=\emptyset$ and $\supp g\cap U_2=\emptyset$;\label{c:3}
    \item for all $U_1,U_2\in\mathcal{U}$, there exists $U_3\in\mathcal{U}_3$ such that $U_1\cup U_2 \subseteq U_3$.\label{c:4}
\end{enumerate}
\end{definition}

Given $x\in [0,1)$, let $T_x^0$ be the open stabilizer subgroup of $T$ consisting of elements that pointwise fix a neighborhood of $x$.  
\begin{lemma}\label{lem:simcom}
For every $x\in [0,1)$, the action $T_x^0\acts \frac{\Rz}{\Zz}\setminus\{x\}$ is compressible.
\end{lemma}

\begin{proof}
We will show that conditions \ref{c:1}-\ref{c:4} from Definition \ref{def:comp} hold with respect to the basis $\mathcal{U}\coloneqq\{(a,b)+\mathbb{Z}:a,b\in \Zz[1/2],x<a<b<x+1\}$ for $\frac{\Rz}{\Zz}\setminus\{x\}$.

\ref{c:1} and \ref{c:4} follow from the definition of $T_x^0$ and density of $\Zz[1/2]$ in $\Rz$.

\ref{c:2}. Let $U_1=(a,b)$ and $U_2=(c,d)$ be elements of $\mathcal{U}$. Take $y_1,y_2\in\Zz[1/2]$ such that
\[x<y_1<\min\{a,c\} <\max\{b,d\}<y_2<x+1.\]
By Lemma \ref{tecint}, there exist piecewise linear bijections $\varphi_1\colon[y_1,a]\to[y_1,c]$, $\varphi_2\colon[a,b]\to[c,d]$ and $\varphi_3\colon[b,y_2]\to[d,y_2]$ with finitely many breakpoints in $\Zz[1/2]$, and slopes in $2^{\Zz}$. Let $g\colon\frac{\Rz}{\Zz}\to \frac{\Rz}{\Zz}$ be given by
\begin{align*}
g(w)\coloneqq\begin{cases}
w,&\text{if $w\in [x,y_1)$},\\
\varphi_1(w),&\text{if $w\in[y_1,a)$},\\
\varphi_2(w),&\text{if $w\in[a,b)$},\\
\varphi_3(w),&\text{if $w\in[b,y_2)$},\\
w,&\text{if $w\in[y_2,x+1)$}.
\end{cases}
\end{align*}
Then $g\in T_x^0$ and $g(U_1)\subset U_2$.

\ref{c:3}. Let $U_1=(a,b)$, $U_2=(c,d)$ and $U_3=(e,f)$ be elements of $\mathcal{U}$ such that $\overline{U_1}\cap\overline{U_2}=\emptyset$. Suppose first that $d<a$. Take $y_1,y_2\in\Zz[1/2]$ such that
\[
d<y_1<a<\max\{b,f\}<y_2<x+1.
\]
By Lemma \ref{tecint}, there exist piecewise linear bijections \[\varphi_1\colon[y_1,a]\to[y_1,\max\{b,f\}]\] and \[\varphi_2\colon[a,y_2]\to[\max\{b,f\},y_2]\] with finitely many breakpoints in $\Zz[1/2]$, and slopes in $2^{\Zz}$. Let $g\colon\frac{\Rz}{\Zz}\to \frac{\Rz}{\Zz}$ be given by
\begin{align*}
g(w)\coloneqq\begin{cases}
w,&\text{if $w\in [x,y_1)$},\\
\varphi_1(w),&\text{if $w\in[y_1,a)$},\\
\varphi_2(w),&\text{if $w\in[a,y_2)$},\\
w,&\text{if $w\in[y_2,x+1)$}.
\end{cases}
\end{align*}
Then $g\in T_x^0$, $g(U_1)\cap U_3=\emptyset$ and $\supp g\cap U_2=\emptyset$. The case where $b<c$ is handled in a similar way.
\end{proof}

We will also need the following result.
\begin{lemma}\label{lem:sim}
For every $x\in [0,1)\setminus\Omega$, the group $T_x^0$ is simple.

\end{lemma}
\begin{proof}
Take sequences of dyadic rationals $(y_n),(z_n)\subseteq[0,1)$ such that $y_n\uparrow x$ and $z_n\downarrow x$. Let $T_n$ be the subgroup of $T$ consisting of elements which pointwise fix a neighborhood of the interval $[y_n,z_n]$. 

Given $n\in\Nz$, let $\varphi_n\colon[0,1)\to [z_n,y_n+1)$ be a piecewise linear map with finitely many breakpoints in $\Omega$, and slopes in $2^\Zz$. Given $g\in [F,F]$, let $\tilde{g}\in T_n$ be given by 
\begin{align*}
\tilde{g}(w)=\begin{cases}w,&\text{if $w\in [y_n,z_n)$},\\
(\varphi_n\circ g\circ\varphi_n^{-1})(w),&\text{if $w\in[z_n,y_n+1)$}.
\end{cases}
\end{align*}
Then, the map $g\mapsto \tilde{g}$ is an isomorphism between $[F,F]$ and $T_n$, whose inverse takes $h\in T_n$, and maps it into $\varphi_n^{-1}\circ h\circ\varphi_n$.
Since $[F,F]$ is simple (\cite[Theorem 4.5]{CFP96}), we conclude that $T_x^0=\bigcup_n T_n$ is simple.
\end{proof}
Given a $C^*$-algebra $A$, we denote by $T(A)$ the set of tracial states on $A$, and by $\partial T(A)$ the set of extreme points of $T(A)$.
\begin{remark}\label{rem:ts}
Recall that a tracial state on a unital $C^*$-algebra is extremal if and only if its GNS representation is factorial (see the proof of \cite[Proposition 11.C.3]{BdlH20}, or \cite[Proposition 1]{HM08}). In particular, if $\pi\colon B\to A$ is a surjective $*$-homomorphism between unital $C^*$-algebras, and $\tau\in\partial T(A)$, then $\tau\circ\pi\in\partial T(B)$.
\end{remark}
\begin{theorem}\label{thm:th}
 The reduced C*-algebra of $T$ has a unique maximal ideal.
\end{theorem}
\begin{proof}
If $F$ is not amenable, then $T$ is C*-simple by \cite[Corollary 4.2]{LBMB18}, so that $C^*_r(T)$ has a unique maximal ideal. 

Assume that $F$ is amenable and fix $x \in [0,1)\setminus\Omega$. 
It follows from \cite[Theorem 4.11]{LBMB18} that $T_x^0$ is the stabilizer of a point in the Furstenberg boundary $\partial_F T$ and, in particular, $T_x^0$ is amenable.

Given a maximal ideal $I$ in $C^*_r(T)$, use Theorem \ref{thm:mi} to find a maximal ideal $J$ of $C^*(T_x^0)$ such that $I=\Ind_{T_x^0}(J)$, and let $\pi\colon C^*(T_x^0)\to C^*(T_x^0)/J$ be the canonical quotient map.  

By considering the action of the amenable group $T_x^0$ on the state space of $C^*(T_x^0)/J$, we obtain that $T(C^*(T_x^0)/J)\neq\emptyset$. Let $\tau\in\partial T(C^*(T_x^0)/J)$. By Remark \ref{rem:ts}, $\tau\circ\pi\in \partial T(C^*(T_x^0))$. By Lemmas \ref{lem:simcom} and \ref{lem:sim}, and \cite[Theorem 2.9]{DM14}, $\tau\circ\pi$ is the character coming from the trivial representation of $C^*(T_x^0)$. Since $J$ is maximal, we conclude that $J$ is the kernel of the trivial representation.
\end{proof}

Given a unitary representation $\pi$ of a group $G$, let $C^*_\pi(G)$ denote the C*-algebra generated by the image of $\pi$. 
Recall that, if $\sigma$ is another unitary representation of $G$, then $\pi$ is \emph{weakly contained} in $\sigma$ (denoted by $\pi\prec\sigma$), if there is a (necessarily surjective) $*$-homomorphism $C^*_\sigma(G)\to C^*_\pi(G)$ which, for each $g\in G$, maps $\sigma_g$ to $\pi_g$.

Given a subgroup $H \leq G$, let $\lambda_{G/H}\colon G\to B(\ell^2(G/H))$ be the \emph{quasi-regular representation} given by $\lambda_{G/H}(g)\delta_{kH}=\delta_{gkH}$, for all $g\in G$ and $kH\in G/H$. Recall that $H$ is amenable if and only if $\lambda_{G/H}$ is weakly contained in $\lambda_G$. Moreover, recall that $H$ is \emph{co-amenable} in $G$ if there is a $G$-invariant mean on $\ell^\infty(G/H)$. Note that every subgroup of an amenable group $G$ is co-amenable in $G$.

\begin{remark}\label{rem:tf}
The C*-algebra $C^*_{\lambda_{T/F}}(T)$ is simple and does not have any tracial states by \cite[Theorem 6.1]{KS22}. 
It follows from Theorem \ref{thm:th} and the fact that $C^*_{\lambda_{T/F}}(T)$ is simple that, if $F$ is amenable, then the unique nonzero simple quotient of $C^*_r(T)$ is $C^*_{\lambda_{T/F}}(T)$.
\end{remark}

The next result strengthens the fact shown in \cite[Theorem 5.5]{HO17} that, if $F$ is amenable, then $C^*_r(T)$ is not simple. The proof is an adaptation of an idea from \cite[Example 6.4]{KS22}. Also compare this result with \cite{GM24}, where, for certain groups $G$, quasi-regular representations are used to produce uncountably many ideals of $C^*(G)$ contained in the kernel of the canonical quotient $C^*(G)\to C^*_r(G)$.

\begin{theorem}\label{thm:infinite-ideals}
    If $F$ is amenable, then $C^*_r(T)$ has infinitely many ideals.
\end{theorem}

\begin{proof}
Given $g\in T$, let $D_+(g)\colon\Omega\to\Zz$ be given by $D_+(g)(x)=n$ if $\lim_{y\to x^+}g'(y)=2^n$, for all $x\in\Omega$. 
Define $D_-$ analogously, so that $D_+$ and $D_-$ are the maps obtained by taking right and left derivatives, respectively, followed by $\log_2$. 
Note that, if $g,h\in T$ and $x\in\Omega$, then $D_\pm(gh)(x)=D_\pm(g)(h(x))+D_\pm(h)(x)$. 

Given $n\in \Zz_{>0}$, consider the action of $T$ on $\Omega\times (\Zz_n)^2$ given by
\[
g(x,i,j) = \big(g(x),i+D_-(g)(x),j+D_+(g)(x)\big),
\]
for all $g\in T$, and $(x,i,j)\in \Omega\times(\Zz_n)^2$.
This action is transitive and the stabilizer of $(0,0,0)$ is $F_n\coloneqq\{g\in F:D_\pm(g)(0)\in n\Zz\}$. Therefore, $\Omega\times(\Zz_n)^2\cong T/F_n$.
 
Assume $F$ is amenable. Then, each $F_n$ is amenable, so that $\lambda_{T/F_{n}}\prec\lambda_T$ for all $n$. Given $m,n\in\Zz_{>0}$ such that $m|n$ and $m<n$, we have that $F_n\leq F_m$. Since $F_m$ is amenable, $F_n$ is co-amenable in $F_m$, so that, by \cite[Proposition 2.3]{KS22}, we have $\lambda_{T/F_{m}}\prec\lambda_{T/F_{n}}$.
 
We claim that $\lambda_{T/F_{m}}\nsucc\lambda_{T/F_{n}}$. Take $g,h\in F$ such that 
\begin{enumerate}
\item $g(x)=x$ for $x\leq 1/2$ and $D_+(g)(1/2)=m$, 
\item$h(x)=x$ for $x\geq 1/2$, $D_-(h)(1/2)=m$. 
\end{enumerate}
Note that the supports of $g$ and $h$ are disjoint with respect to the action of $T$ on $\Omega\times(\Zz_m)^2\cong T/F_m$, but not with respect to the action of $T$ on $\Omega\times(\Zz_n)^2\cong T/F_n$. By \cite[Proposition 6.3]{KS22}, 
\[
\lambda_{T/F_{m}}(gh)-\lambda_{T/F_{m}}(g)-\lambda_{T/F_{m}}(h)=1\neq \lambda_{T/F_{n}}(gh)-\lambda_{T/F_{n}}(g)-\lambda_{T/F_{n}}(h).
\]
In particular, there cannot be a $*$-homomorphism $C^*_{\lambda_{T/F_m}}(T)\to C^*_{\lambda_{T/F_n}}(T)$ which, for each $t\in T$, maps $\lambda_{T/F_m}(t)$ to $\lambda_{T/F_n}(t)$. This shows the claim.  
\end{proof}
Theorems \ref{thm:th} and \ref{thm:infinite-ideals}, together with Remark \ref{rem:tf}, \cite[Theorem 5.5]{HO17} and \cite[Corollary 4.2]{LBMB18}, yield the following characterizations of the amenability of $F$ in terms of the ideal structure of $C^*_r(T)$. 
\begin{corollary}[{\cite{HO17},\cite{LBMB18},\cite{KS22}}]\label{cor:is}
    The following statements are equivalent:
    \begin{enumerate}[\upshape(1)]
        \item Thompson's group $F$ is amenable;
        \item Thompson's group $T$ is not $C^*$-simple;
        \item the unique maximal ideal of $C^*_r(T)$ is nonzero;
        \item the unique nonzero simple quotient of $C^*_r(T)$ is traceless;
        \item $C^*_r(T)$ has infinitely many ideals.
    \end{enumerate}
\end{corollary}

\end{document}